\documentclass[12pt,reqno]{amsart}
\usepackage{amsmath,amssymb,extarrows}
\usepackage{url}
\usepackage{tikz,enumerate}
\usepackage{diagbox}
\usepackage{appendix}
\usepackage{epic}
\usepackage{comment}

\usepackage{float} %support many floats

\usepackage{cite}

\usepackage{hyperref}
\usepackage{array}

\usepackage{booktabs}

\allowdisplaybreaks[4]

\newtheorem{theorem}{Theorem}

\newtheorem{lemma}[theorem]{Lemma}

\theoremstyle{definition}

\theoremstyle{remark}

\numberwithin{equation}{section}
\numberwithin{theorem}{section}
\numberwithin{defn}{section}
\DeclareMathOperator{\sg}{sg}

\usepackage{graphicx} % Required for inserting images

\begin{document}
\title[Proofs of two conjectural identities on partial Nahm sums]{Proofs of two conjectural identities on partial Nahm sums}
\author{Changsong Shi and Liuquan Wang}

\address[C.\ Shi]{School of Mathematics and Statistics, Wuhan University, Wuhan 430072, Hubei, People's Republic of China}
\email{changsong@whu.edu.cn}

\address[L.\ Wang]{School of Mathematics and Statistics, Wuhan University, Wuhan 430072, Hubei, People's Republic of China}
\email{wanglq@whu.edu.cn;mathlqwang@163.com}



\begin{abstract}
Recently, Wang and Zeng investigated modularity of partial Nahm sums and discovered 14 modular families of such sums. They confirmed modularity for 13 families and proposed a conjecture consisting of two Rogers--Ramanujan type identities for the remaining family. We prove these conjectural identities in two steps. First, employing a transformation formula involving two Bailey pairs, we transform the partial Nahm sums into some specific Hecke-type series. Second, using two distinct approaches, we convert these Hecke-type series to the desired modular infinite products.
\end{abstract}
\maketitle

\section{Introduction and main results}
The famous Rogers--Ramanujan identities give beautiful factorization of two particular $q$-hypergeometric series:
\begin{align}\label{RR}
\sum_{n=0}^\infty \frac{q^{n^2}}{(1-q)(1-q^2)\cdots (1-q^n)}&=\prod\limits_{n=0}^\infty \frac{1}{(1-q^{5n+1})(1-q^{5n+4})}, \\
\sum_{n=0}^\infty \frac{q^{n^2+n}}{(1-q)(1-q^2)\cdots (1-q^n)}&=\prod\limits_{n=0}^\infty \frac{1}{(1-q^{5n+2})(1-q^{5n+3})}.
\end{align}
They were first discovered by Rogers \cite{Rogers1894} and later rediscovered by Ramanujan. Since then lots of similar sum-to-product $q$-series identities have been found and they were usually referred as Rogers--Ramanujan type identities. Such identities have important applications in combinatorics, number theory, Lie algebra and mathematical physics, etc.

One of the motivation to study Rogers--Ramanujan type identities comes from Nahm's problem \cite{Nahm07}.  The problem is to find all positive definite $r\times r$ rational matrix $A$, $r$-dimensional rational column vector $B$ and rational scalar $C$ such that the Nahm sum
\begin{align}\label{eq-full-Nahm}
f_{A,B,C}(q):=\sum_{n=(n_1,\dots,n_r)^\mathrm{T} \in \mathbb{Z}^r} \frac{q^{\frac{1}{2}n^\mathrm{T} An+n^\mathrm{T} B+C}}{(q;q)_{n_1}\cdots (q;q)_{n_r}}
\end{align}
is modular.  Here and throughout this paper, we adopt standard $q$-series notations:
\begin{align}
    (a;q)_\infty&:=\prod\limits_{k=0}^{\infty }(1-aq^k), \quad |q|<1, \\
    (a;q)_n&:=\frac{(a;q)_\infty}{(aq^n;q)_\infty}, \quad n\in \mathbb{Z}.
\end{align}
As a convention, we agree that $1/(q;q)_n=0$ for $n\in \mathbb{Z}_{<0}$. Sometimes we also use compressed notation
\begin{align}
    (a_1,a_2,\dots,a_m;q)_n:=\prod\limits_{k=1}^m(a_k;q)_n, \quad n \in \mathbb{Z}\cup \{\infty\}.
\end{align}
As usual, when $f_{A,B,C}(q)$ is modular we call $(A,B,C)$ a modular triple. Those modular Nahm sums are usually expected to be characters of certain 2-dimensional rational conformal field theories.

Around 2007, Zagier \cite{Zagier} studied Nahm's problem systematically. He proved that there are exactly seven rank one modular triples. He also provided 11 and 12 possible modular triples in rank two and three cases, respectively. Now the modularity of these candidates have all been confirmed mainly by the works of Zagier \cite{Zagier}, Vlasenko--Zwegers \cite{VZ}, Cherednik--Feigin \cite{Feigin}, Wang \cite{Wang-rank2,Wang-rank3} and Cao--Rosengren--Wang \cite{Cao-Rosengren-Wang}. The main idea in these works except \cite{Feigin} are to express Nahm sums as modular infinite products by establishing Rogers--Ramanujan type identities. In contrast, Cherednik--Feigin \cite{Feigin} proved the modularity of two rank two Nahm sums through nilpotent double affine Hecke algebras.

Recently, Wang and Zeng \cite{Wang-Zeng} extended Nahm's problem to partial Nahm sums. Let $A$ be a $r\times r$  nonzero rational symmetric matrix, $B$ a $r$-dimensional column vector, and $C$ a scalar. For any lattice $L\subseteq \mathbb{Z}^r$ and coset $v+L\in \mathbb{Z}^r/L$, they defined Nahm sum on lattice coset as
\begin{align}\label{eq-lattice-sum}
f_{A,B,C,v+L}(q):=\sum_{n=(n_1,\dots,n_r)^\mathrm{T} \in v+L} \frac{q^{\frac{1}{2}n^\mathrm{T} An+n^\mathrm{T} B+C}}{(q;q)_{n_1}\cdots (q;q)_{n_r}}
\end{align}
subject to the condition that the infinite series is absolutely convergent. Note that the sum is actually taken over the set $(v+L)\cap \mathbb{Z}_{\geq 0}^r$. In particular, when $L=\mathbb{Z}^r$,  $f_{A,B,C,v+L}(q)=f_{A,B,C}(q)$ is the full Nahm sum. Wang and Zeng \cite{Wang-Zeng} then proposed the problem to find all suitable $(A,B,C,v+L)$ such that $f_{A,B,C,v+L}(q)$ is modular.

In the case of rank $r=1$, Wang and Zeng \cite{Wang-Zeng} found eight modular partial Nahm sums by some known Rogers--Ramanujan type identities. In the case of rank $r=2$ and $L$ is one of the lattices $\mathbb{Z}(2,0)+\mathbb{Z}(0,1)$, $\mathbb{Z}(1,0)+\mathbb{Z}(0,2)$ or $\mathbb{Z}(2,0)+\mathbb{Z}(0,2)$, they found 14 types of symmetric matrices $A$ so that the partial Nahm sums $f_{A,B,C,v+L}(q)$ are modular for some vectors $B,v$ and scalars $C$. They established various Rogers--Ramanujan type identities to confirm the modularity of 13 families of these partial Nahm sums. For example, they proved the following identity \cite[Theorem 1.3]{Wang-Zeng}:
\begin{align}\label{eq-WZ}
       \sum_{i,j\ge0}\frac{q^{4ij+i+3j}}{(q;q)_{2i+1}(q;q)_{2j}}&=\frac{(q^8;q^8)_\infty}{(q;q^2)_\infty^2(q^4;q^8)_\infty}.
\end{align}
This implies that the partial Nahm sum $f_{A,B,C,v+L}(q)$ is modular for
\begin{equation}
\begin{split}
A=\begin{pmatrix} 0 & 1 \\  1  & 0 \end{pmatrix}, ~~  B=(1/2,1/2)^\mathrm{T}, ~~ C=-5/12,
 ~~ L=\mathbb{Z}(2,0)^\mathrm{T}+\mathbb{Z}(0,2)^\mathrm{T}
\end{split}
\end{equation}
and vector $v=(1,0)^\mathrm{T}$. After interchanging $i$ with $j$ this also confirms the modularity of $f_{A,B,C,v'+L}(q)$ with the vector $v$ replaced by $v'=(0,1)^\mathrm{T}$.

There is only one family of partial Nahm sums in the list of Wang and Zeng \cite{Wang-Zeng} whose modularity remains open. They are partial Nahm sums $f_{A,B_i,C_i,v_i+L}(q)$ correspond to
\begin{equation}\label{eq-quadruple}
\begin{split}
&A=\begin{pmatrix}
0 & 1/2 \\ 1/2 & 0
\end{pmatrix}, ~~L=\mathbb{Z}(2,0)^\mathrm{T}+\mathbb{Z}(0,2)^\mathrm{T}, \\
&B_1=(1/2,1/2)^\mathrm{T},~~ C_1=1/20, ~~v_1=(0,0)^\mathrm{T}, \\
&B_2=(0,1)^\mathrm{T}, ~~C_2=-21/20, ~~v_2=(1,0)^\mathrm{T}, \\
&B_3=(1,0)^\mathrm{T}, ~~ C_3=-21/20, ~~v_3=(0,1)^\mathrm{T}.
\end{split}
\end{equation}
Wang and Zeng \cite[Conjecture 3.3]{Wang-Zeng} proposed two conjectural identities (see \eqref{eq-conj-1} and \eqref{eq-conj-2}) which express these partial Nahm sums as modular infinite products. We will confirm their conjecture and hence we state it as a theorem.
\begin{theorem}\label{thm-main}
We have
    \begin{align}
        \sum_{i,j\ge0}\frac{q^{2ij+i+j}}{(q;q)_{2i}(q;q)_{2j}}&=\frac{1}{(q;q^2)_\infty^2(q^{2},q^{8};q^{10})_\infty},\label{eq-conj-1}\\
        \sum_{i,j\ge0}\frac{q^{2ij+i+3j}}{(q;q)_{2i+1}(q;q)_{2j}}&=\frac{1}{(q;q^2)_\infty^2(q^{4},q^{6};q^{10})_\infty}. \label{eq-conj-2}
    \end{align}
\end{theorem}

The identities \eqref{eq-conj-1} and \eqref{eq-conj-2} look quite similar to \eqref{eq-WZ}, and one may wonder whether they can be proved in the same way. Wang and Zeng \cite{Wang-Zeng} were able to reduce the double sums in these identities to some single sums. However, unlike \eqref{eq-WZ}, it seems difficult to further convert the single sums reduced from \eqref{eq-conj-1} and \eqref{eq-conj-2} (see \cite[Eqs.\ (3.14) and (3.15)]{Wang-Zeng}) to infinite products. Here we use a different approach to prove \eqref{eq-conj-1} and \eqref{eq-conj-2}. Our proof consists mainly of two steps. For the first step, we will use the Bailey machinery (especially Lovejoy's transformation formula \cite[Theorem 1.2]{Lovejoy})  to transform the partial Nahm sums into certain Hecke-type series of the form
\begin{align}\label{f-defn}
f_{a,b,c}(x,y,q):=\sum_{\sg(r)=\sg(s)}\sg(r)(-1)^{r+s}x^ry^sq^{a\binom{r}{2}+brs+c\binom{s}{2}}.
\end{align}
Here $x,y\in \mathbb{C}^{*}:=\mathbb{C} \backslash \{0\}$ and $\sg(r):=1$ for $r\geq 0$ and $\sg(r):=-1$ for $r<0$. For the second step we have two different proofs. The first proof is to use the techniques developed in the work of Hickerson and Mortenson \cite{Hickerson-Mortenson} to further express these Hecke-type series as Appell--Lerch sums
\begin{align}\label{m-defn}
m(x,q,z):=\frac{1}{j(z;q)}\sum_{r=-\infty}^{\infty} \frac{(-1)^rq^{\binom{r}{2}}z^r}{1-q^{r-1}xz},
\end{align}
where $x,z\in \mathbb{C}^{*}$ with neither $z$ nor $xz$ an integral power of $q$ and
\begin{align}\label{j-defn}
    j(z;q):=(z,q/z,q;q)_\infty.
\end{align}
Then we will be able to convert these expressions to the conjectured infinite products. The second proof is to convert the Hecke-type series obtained in the first step to some other Hecke-type series whose modularity is known.

We remark here that our proofs are quite different from proofs for many other Rogers--Ramanujan type identities in the literature. While in \cite{Wang-rank2,Wang-rank3,Wang-Zeng} usually only one Bailey pair is invoked per time, here for the first step we used two Bailey pairs simultaneously.  Another distinction arises from the second step. As in most cases, after using Bailey pairs we usually arrive at some theta functions which can be easily expressed as infinite products and are clearly modular. In contrast, this time we only obtain Hecke-type series and need rather more complicated techniques to convert them to modular infinite products. It is well-known that Hecke-type series and Appell--Lerch sums play a prominent role in the theory of mock theta functions. However, it is relatively rare to see their usage in proving Rogers--Ramanujan type identities.

The rest of this paper is organized as follows. In Section \ref{sec-pre} we introduce some notation and basic knowledge about Bailey pairs as well as properties about Hecke-type series and Appell--Lerch sums. Section \ref{sec-proof} is devoted to giving proofs for Theorem \ref{thm-main}. We will first transform the double sums to some Hecke-type series, and then we present two different proofs to show that these Hecke-type series agree with the desired modular infinite products.

\section{Preliminaries}\label{sec-pre}

The Jacobi triple product identity asserts that
\begin{align}\label{eq-Jacobi}
j(z;q) =\sum_{n=-\infty}^{\infty} (-1)^nq^{\binom{n}{2}}z^n.
\end{align}
As some special cases, we let $a$ and $m$ be rational numbers with $m$ positive and define
\begin{align*}
J_{a,m}:=j(q^a;q^m), \quad \overline{J}_{a,m}:=j(-q^a;q^m) \,\, \textrm{and} \,\, J_{m}:=J_{m,3m}=(q^m;q^m)_\infty.
\end{align*}

If $(\alpha_n(a;q),\beta_n(a;q))$ satisfies
    \begin{align}\label{BP-defn}
        \beta_n(a;q)=\sum_{i=0}^n\frac{\alpha_i(a;q)}{(q;q)_{n-i}(aq;q)_{n+i}}, \quad \forall n\geq 0,
    \end{align}
then we say that $(\alpha_n(a;q),\beta_n(a;q))$ is a Bailey pair relative to $(a;q)$.

For the purpose of proving Theorem \ref{thm-main}, we need three Bailey pairs.
\begin{lemma}\label{3bailey}
The following $(\alpha_n(a;q),\beta_n(a;q))$ are Bailey pairs relative to $(a;q)$ ($a\in \{1,q,q^2\}$):
    \begin{align}
    &\begin{cases}
        \beta_n(1;q)=\frac{1}{(q;q)_{2n}}, \\
        \alpha_0(1;q)=1,  \label{BP1} \\ \alpha_{3n-1}(1;q)=-q^{6n^2-5n+1},\ \alpha_{3n}(1;q)=q^{6n^2-n}+q^{6n^2+n},\ \alpha_{3n+1}(1;q)=-q^{6n^2+5n+1};
        \end{cases}\\
    &\begin{cases}
        \beta_n(q;q)=\frac{1}{(q;q)_{2n}}, \\
        \alpha_0(q;q)=1, \label{BP2}\\
        \alpha_{3n-1}(q;q)=\frac{-q^{6n^2-5n+1}+q^{6n^2+n}}{1-q},\ \alpha_{3n}(q;q)=\frac{q^{6n^2-n}-q^{6n^2+5n+1}}{1-q},\ \alpha_{3n+1}(q;q)=0;
    \end{cases}\\
    &\begin{cases}
        \beta_n(q^2;q)=\frac{1}{(q;q)_{2n+1}}, \\
        \alpha_0(q^2;q)=\frac{1}{1-q},\\
        \alpha_{3n-1}(q^2;q)=0,\ \alpha_{3n}(q^2;q)=\frac{q^{6n^2+n}-q^{6n^2+7n+2}}{(1-q)(1-q^2)},\ \alpha_{3n+1}(q^2;q)=\frac{-q^{6n^2+5n+1}+q^{6n^2+11n+5}}{(1-q)(1-q^2)}. \label{BP3}
    \end{cases}
    \end{align}
\end{lemma}
The Bailey pair \eqref{BP1} appeared as the A(1) Bailey pair in Slater's list \cite[p.~463]{Slater}. We did not find a reference for \eqref{BP2} and \eqref{BP3}, but we can derive them directly from Slater's results in the following way.
\begin{proof}
Recall that Slater justified the Bailey pair \eqref{BP1} through the identity \cite[Eq.\ (3.3)]{Slater}:
\begin{align}\label{Slater-id-1}
   \frac{1}{(q;q)_{2n}}=\sum_{r=-[n/3]}^{[n/3]}\frac{(1-q^{6r+1})q^{6r^2-r}}{(q;q)_{n-3r}(q;q)_{n+3r+1}}.
\end{align}
This can  be rewritten as
\begin{align}
  \frac{1}{(q;q)_{2n}}&=\sum_{r=0}^{[n/3]}\frac{(1-q^{6r+1})q^{6r^2-r}}{(1-q)(q;q)_{n-3r}(q^2;q)_{n+3r}} \nonumber \\
  &\quad +\sum_{r=1}^{[n/3]}\frac{(1-q^{-6r+1})q^{6r^2+r}}{(1-q)(q;q)_{n-(3r-1)}(q^2;q)_{n+3r-1}}.
\end{align}
Comparing this with the definition \eqref{BP-defn}, we see that \eqref{BP2} is indeed a Bailey pair.

Recall another identity from \cite[Eq.\ (3.6)]{Slater}:
\begin{align}\label{Slater-id-2}
    \sum_{r=-[n/3]}^{[n/3]}\frac{(1-q^{6r+2})q^{6r^2+r}}{(q^2;q)_{n+3r+1}(q;q)_{n-3r}}=\frac{1}{(q^2;q)_{2n}}.
\end{align}
Multiplying both sides by $1/(1-q)$, we rewrite it as
\begin{align}
\frac{1}{(q;q)_{2n+1}}&=\sum_{r=0}^{[n/3]}\frac{(1-q^{6r+2})q^{6r^2+r}}{(1-q)(1-q^2)(q;q)_{n-3r}(q^3;q)_{n+3r}} \nonumber \\
&\quad +\sum_{r=1}^{[n/3]}\frac{(1-q^{-6r+2})q^{6r^2-r}}{(1-q)(1-q^2)(q;q)_{n-(3r-2)}(q^3;q)_{n+3r-2}}.
\end{align}
Comparing this with the definition \eqref{BP-defn}, we see that \eqref{BP3} is indeed a Bailey pair.
\end{proof}

The following result found from Lovejoy's work \cite{Lovejoy} will play a key role in the first step of proving Theorem \ref{thm-main}.
\begin{lemma}\label{lem-lovejoy}
    (Cf.\ \cite[Theorem\ 1.2]{Lovejoy}.) If $(\alpha_n(a;q), \beta_n(a;q))$ is a Bailey pair relative to $(a, q)$ and $(\alpha_n^{\prime}(z;q), \beta_n^{\prime}(z;q))$ is a Bailey pair relative to $(z,q)$, then
    \begin{align}
        &\sum_{n, s \geq 0} a^s z^n q^{2 n s+s} \beta_n(a;q) \beta_s^{\prime}(z;q)\nonumber\\
        &=\frac{1}{(a q, z;q)_{\infty}} \sum_{n, r \geq 0} \frac{a^n z^r q^{2 n r+n}(1-z)}{1-z q^{2 n}} \alpha_r(a;q) \alpha_n^{\prime}(z;q).
    \end{align}
\end{lemma}

As a key to the second step of proving Theorem \ref{thm-main}, we will use some useful properties about the functions $f_{a,b,c}(x,y,q)$ and $m(x,q,z)$ from the work of Hickerson and Mortenson \cite{Hickerson-Mortenson}. It is straightforward to see by definition that
\begin{align}
    f_{a,b,c}(x,y,q^m)&=f_{am,bm,cm}(x,y,q),~~ m\in \mathbb{Z}_{>0},\label{f-base} \\
    f_{a,b,a}(x,y,q)&=f_{a,b,a}(y,x,q). \label{f-symmetry}
\end{align}

Following \cite{Hickerson-Mortenson}, the term ``generic'' will be used to mean that the parameters do not cause poles in the Appell--Lerch sums or in the quotients of theta functions.
\begin{lemma}\label{f-m-expan}
(Cf.\ \cite[Corollary 8.2]{Hickerson-Mortenson}.) Let $\ell\in \mathbb{Z}$. For generic $x,y\in \mathbb{C}^*$
\begin{align}\label{eq-f232}
f_{2,3,2}(x,y,q)=&\sum_{r=0}^1 \Big[\frac{x^r}{q^{r^2}y^r}j(q^ry;q^2)m\Big(\frac{q^{6-5r}x^2}{y^3},q^{10},\frac{q^{2\ell}y^2}{x^2}\Big)\nonumber \\
&+\frac{y^r}{q^{r^2}x^r}j(q^rx;q^2)m\Big(\frac{q^{6-5r}y^2}{x^3},q^{10},\frac{x^2}{q^{2\ell}y^2}\Big)\Big].
\end{align}
\end{lemma}

\begin{lemma}\label{lem-f-id}
(Cf.\ \cite[Proposition 6.2 and Corollary 6.4]{Hickerson-Mortenson}.) We have
\begin{align}
f_{a,b,c}(x,y,q)&=-\frac{q^{a+b+c}}{xy}f_{a,b,c}(q^{2a+b}/x,q^{2c+b}/y,q), \label{f-id-0} \\
    f_{a,b,c}(x,y,q)&=-yf_{a,b,c}(q^bx,q^cy,q)+j(x;q^a), \label{f-id-1} \\
    f_{a,b,c}(x,y,q)&=-xf_{a,b,c}(q^ax,q^by,q)+j(y;q^c). \label{f-id-2}
\end{align}
\end{lemma}

\begin{lemma}\label{lem-m-prop}
(Cf.\ \cite[Proposition\ 3.1]{Hickerson-Mortenson}.) For generic $x,z\in \mathbb{C}^{*}$
\begin{align}
m(x,q,z)&=m(x,q,qz), \label{m-id-1} \\
m(x,q,z)&=x^{-1}m(x^{-1},q,z^{-1}), \label{m-id-2} \\
m(qx,q,z)&=1-xm(x,q,z). \label{m-id-3}
\end{align}
\end{lemma}
From \eqref{m-id-2} and \eqref{m-id-3} we deduce that
\begin{align}
m(x,q,z^{-1})=1-m(q/x,q,z). \label{m-id-4}
\end{align}

\begin{lemma}\label{lem-m-minus}
(Cf.\ \cite[Theorem 3.3]{Hickerson-Mortenson}.) For generic $x,z_0,z_1\in \mathbb{C}^{*}$
\begin{align}
m(x,q,z_1)-m(x,q,z_0)=\frac{z_0J_1^3j(z_1/z_0;q)j(xz_0z_1;q)}{j(z_0;q)j(z_1;q)j(xz_0;q)j(xz_1;q)}. \label{m-minus}
\end{align}
\end{lemma}

\section{Proofs of Theorem \ref{thm-main}}\label{sec-proof}

To make our proof more readable, we extract some steps as separate lemmas. Throughout this section we denote
\begin{align}
S(q)&:=(q;q)_\infty^2 \sum_{i,j\geq 0} \frac{q^{2ij+i+j}}{(q;q)_{2i}(q;q)_{2j}},\\
T(q)&:=(q;q)_{\infty}^2\sum_{i,j\geq 0}\frac{q^{2ij+i+3j}}{(q;q)_{2i+1}(q;q)_{2j}}.
\end{align}
Then \eqref{eq-conj-1} and \eqref{eq-conj-2} are equivalent to
\begin{align}\label{S-product}
    S(q^{\frac{1}{2}})=J_1J_{2,5}
\end{align}
and
\begin{align}\label{T-product}
    T(q^{\frac{1}{2}})=J_1J_{1,5},
\end{align}
respectively.

We first express $S(q)$ and $T(q)$ as Hecke-type series.
\begin{lemma}\label{lem-ST}
We have
\begin{align}
S(q^{\frac{1}{2}})&=f_{2,3,2}(-q^4,-q^4,q^3)+2f_{2,3,2}(-q^4,-q^5,q^3) +f_{2,3,2}(-q^5,-q^8,q^3) \nonumber \\
&\quad \quad  -\overline{J}_{1,6}-2\overline{J}_{2,6}, \label{S-lem} \\
T(q^{\frac{1}{2}})&=f_{2,3,2}(-q^7,-q^5,q^3)-qf_{2,3,2}(-q^7,-q^8,q^3)+q^{-1}f_{2,3,2}(-q,-q^4,q^3) \nonumber \\
    &\quad \quad +f_{2,3,2}(-q^5,-q^5,q^3)-(1+q^{-1})\overline{J}_{1,6}. \label{T-lem}
\end{align}
\end{lemma}
\begin{proof}
(1) Applying  Lemma \ref{lem-lovejoy} with the two Bailey pairs in \eqref{BP1} and \eqref{BP2}, we deduce that
\begin{align}\label{S-begin}
 &S(q)=(q;q)_{\infty}^2\sum_{n,s\geq 0}q^{2ns+s+n}\beta_n(1;q)\beta_s(q;q) \nonumber \\
&=\sum_{n,r\geq 0}\frac{q^{2nr+n+r}(1-q)\alpha_r(1;q)\alpha_n(q;q)}{1-q^{2n+1}}.
\end{align}
We decompose $S(q)$ into three parts according to $n=0$, $n\geq 1$ and $r=0$, and $n,r\geq 1$, respectively. Namely, we define
\begin{align}
    A(q)&:=\sum_{r\geq 0}q^r\alpha_r(1;q),\\
    B(q)&:=\sum_{n\geq 1}\frac{q^n(1-q)\alpha_n(q;q)}{1-q^{2n+1}},\\
    F_{a,b}(q)&:=\sum_{\substack{n,r\geq 1\\ n\equiv a\ (\mathrm{mod}\ 3)\\ r\equiv b\ (\mathrm{mod}\ 3)}}\frac{q^{2nr+n+r}(1-q)\alpha_n(q;q)\alpha_r(1;q)}{1-q^{2n+1}}.
\end{align}
Then \eqref{S-begin} can be written as
\begin{align}\label{S=A+B+F}
    S(q)=A(q)+B(q)+\sum_{0\leq a,b\leq 2}F_{a,b}(q).
\end{align}

Now we calculate the series  $A(q)$, $B(q)$ and $F_{a,b}(q)$ separately. We have
\begin{align}\label{A-exp}
    &A(q)=1+\sum_{r\geq 1}q^{3r}(q^{6r^2+r}+q^{6r^2-r})+\sum_{r \geq 1}q^{3r-1}(-q^{6r^2-5r+1})+\sum_{r\geq 0}q^{3r+1}(-q^{6r^2+5r+1})\nonumber\\
    &=1+\sum_{r\geq 1}\Big(q^{6r^2+4r}+q^{6r^2+2r}-q^{6r^2-2r}-q^{6r^2-4r}\Big)
\end{align}
where we replaced $r$ by $r-1$ for the last term.

For the series $B(q)$, we have
\begin{align}\label{B-exp}
    &B(q)=\sum_{n\geq1}\frac{q^n(1-q)\alpha_n(q;q)}{1-q^{2n+1}}=\sum_{n\geq 1}\Big(\frac{q^{6n^2+2n}-q^{6n^2+8n+1}}{1-q^{6n+1}}+\frac{-q^{6n^2-2n}+q^{6n^2+4n-1}}{1-q^{6n-1}}\Big)\nonumber\\
    &=\sum_{n\geq 1}\big(q^{6n^2+2n}-q^{6n^2-2n}\big).
\end{align}

For the series $F_{a,b}(q)$, we have
\begin{align}
 &F_{1,0}(q)=F_{1,1}(q)=F_{1,2}(q)=0, \label{F-zero} \\
   & F_{0,0}(q)=\sum_{n,r\geq 1}\frac{q^{18nr+3n+3r}(q^{6r^2-r}+q^{6r^2+r})(q^{6n^2-n}-q^{6n^2+5n+1})}{1-q^{6n+1}}\nonumber\\
    &=\sum_{n,r\geq 1}q^{6n^2+6r^2+18nr+2n+2r}(1+q^{2r}), \label{F00}\\
   & F_{0,1}(q)=\sum_{\substack{n,r\geq 1}}\frac{q^{18nr-9n+3r-2}(-q^{6r^2-7r+2})(q^{6n^2-n}-q^{6n^2+5n+1})}{1-q^{6n+1}}\nonumber\\
    &=-\sum_{n,r\geq 1}q^{6n^2+6r^2+18nr-4r-10n}, \label{F01}\\
   & F_{0,2}(q)=\sum_{n,r\geq 1}\frac{q^{18nr-3n+3r-1}(-q^{6r^2-5r+1})(q^{6n^2-n}-q^{6n^2+5n+1})}{1-q^{6n+1}}\nonumber\\
    &=-\sum_{n,r\geq 1}q^{6n^2+6r^2+18nr-2r-4n}, \label{F02}\\
    &F_{2,0}(q)=\sum_{n,r\geq 1}\frac{q^{18nr+3n-3r-1}(q^{6r^2-r}+q^{6r^2+r})(-q^{6n^2-5n+1}+q^{6n^2+n})}{1-q^{6n-1}}\nonumber\\
    &=-\sum_{n,r\geq 1}q^{6n^2+6r^2+18nr-2n-4r}(1+q^{2r}), \label{F10} \\
   & F_{2,1}(q)=\sum_{\substack{n,r\geq 1}}\frac{q^{18nr-9n-3r+1}(-q^{6r^2-7r+2})(-q^{6n^2-5n+1}+q^{6n^2+n})}{1-q^{6n-1}}\nonumber\\
    &=\sum_{n,r\geq 1}q^{6n^2+6r^2+18nr-14n-10r+4}=\sum_{\substack{n\geq 1\\r\geq 0}}q^{6n^2+6r^2+18nr+4n+2r}\nonumber\\
    &=\sum_{n\geq 1}q^{6n^2+4n}+\sum_{n,r\geq 1}q^{6n^2+6r^2+18nr+4n+2r}, \label{F11}\\
   & F_{2,2}(q)=\sum_{n,r\geq 1}\frac{q^{18nr-3n-3r}(-q^{6r^2-5r+1})(-q^{6n^2-5n+1}+q^{6n^2+n})}{1-q^{6n-1}}\nonumber\\
    &=\sum_{n,r\geq 1}q^{6n^2+6r^2+18nr-8r-8n+2}=\sum_{\substack{n\geq 0\\r\geq 1}}q^{6n^2+6r^2+18nr+4n+10r}\nonumber\\
    &=\sum_{r\geq 1}q^{6r^2+10r}+\sum_{n,r\geq 1}q^{6n^2+6r^2+18nr+4n+10r}. \label{F12}
\end{align}
Substituting \eqref{A-exp}--\eqref{F12} into \eqref{S=A+B+F} and replacing $q$ by $q^{\frac{1}{2}}$, we deduce that
\begin{align}\label{S-wS}
S(q^{\frac{1}{2}})=1+\sum_{n=1}^\infty (q^{3n^2+5n}+2q^{3n^2+2n}+2q^{3n^2+n}-2q^{3n^2-n}-q^{3n^2-2n})+\widetilde{S}(q),
\end{align}
where
\begin{align}\label{wS-exp}
&\widetilde{S}(q):=\sum_{n,r\geq 1} q^{3n^2+3r^2+9nr}\Big(q^{n+r}+2q^{n+2r}+q^{2n+5r}-q^{-n-r}-2q^{-n-2r}-q^{-2n-5r}\Big) \nonumber \\
&=\Big(\sum_{n,r\geq 1}-\sum_{n,r\leq -1}\Big) q^{3n^2+9nr+3r^2+n+r}+2\Big(\sum_{n,r\geq 1}-\sum_{n,r\leq -1}\Big)q^{3n^2+9nr+3r^2+n+2r} \nonumber \\
&\quad +\Big(\sum_{n,r\geq 1}-\sum_{n,r\leq -1}\Big)q^{3n^2+9nr+3r^2+2n+5r}  \nonumber \\
&=\Big(f_{2,3,2}(-q^4,-q^4,q^3)-\sum_{n=1}^\infty q^{3n^2+n}-\sum_{r=1}^\infty q^{3r^2+r}-1\Big)\nonumber \\
&\quad +2\Big(f_{2,3,2}(-q^4,-q^5,q^3)-\sum_{n=1}^\infty q^{3n^2+n}-\sum_{r=1}^\infty q^{3r^2+2r}-1\Big) \nonumber \\
&\quad +\Big(f_{2,3,2}(-q^5,-q^8,q^3)-\sum_{n=1}^\infty q^{3n^2+2n}-\sum_{r=1}^\infty q^{3r^2+5r}-1\Big).
\end{align}
Substituting \eqref{wS-exp} into \eqref{S-wS}, we deduce that
\begin{align}\label{S-S1S2}
&S(q^{\frac{1}{2}})=S_1(q)+S_2(q),
\end{align}
where
\begin{align}
&S_1(q)=f_{2,3,2}(-q^4,-q^4,q^3)+2f_{2,3,2}(-q^4,-q^5,q^3)+f_{2,3,2}(-q^5,-q^8,q^3), \label{S1-defn}\\
&S_2(q)=-3-2\sum_{n=1}^\infty (q^{3n^2+n}+q^{3n^2-n})-\sum_{n=1}^\infty (q^{3n^2-2n}+q^{3n^2+2n}) \nonumber \\
&=-2\sum_{n=-\infty}^\infty q^{3n^2+n}-\sum_{n=-\infty}^\infty q^{3n^2+2n}\nonumber\\
&=-2\overline{J}_{2,6}-\overline{J}_{1,6}. \label{S2-defn}
\end{align}
This proves \eqref{S-lem}.

(2) Applying Lemma \ref{lem-lovejoy} with the two Bailey pairs in \eqref{BP3} and \eqref{BP2}, we deduce that
\begin{align}\label{T-start}
   & T(q)=(q;q)_{\infty}^2\sum_{n,s\geq 0}q^{2ns+n+3s}\beta_n(q^2;q)\beta_s(q;q)\nonumber\\
    &=(1-q)(1-q^2)\sum_{n,r\geq 0}\frac{q^{2nr+3n+r}(1-q)\alpha_n(q;q)\alpha_r(q^2;q)}{1-q^{2n+1}}.
\end{align}
We decompose $T(q)$ into three parts according to $n=0$, $n\geq 1$ and $r=0$, and $n,r\geq 1$, respectively. Namely, we define
\begin{align}
    C(q)&:=(1-q)(1-q^2)\sum_{r\geq 0}q^r\alpha_r(q^2;q),\\
    D(q)&:=(1-q)(1-q^2)\sum_{n\geq 1}\frac{q^{3n}\alpha_n(q;q)}{1-q^{2n+1}},\\
    G_{a,b}(q)&:=(1-q)^2(1-q^2)\sum_{\substack{n,r\geq 1\\ n\equiv a\ (\mathrm{mod}\ 3)\\r\equiv b\ (\mathrm{mod}\ 3)}}\frac{q^{2nr+3n+r}\alpha_r(q^2,q)\alpha_n(q;q)}{1-q^{2n+1}}.
\end{align}
Then \eqref{T-start} can be written as
\begin{align}\label{T-CDG}
    T(q)=C(q)+D(q)+\sum_{0\leq a,b\leq 2}G_{a,b}(q).
\end{align}
Now we calculate the series $C(q)$, $D(q)$ and $G_{a,b}(q)$ separately. We have
\begin{align}\label{C-exp}
  &C(q)=1-q^2+\sum_{r\geq 1}(q^{6r^2+4r}-q^{6r^2+10r+2})+\sum_{r\geq 0}(q^{6r^2+14r+6}-q^{6r^2+8r+2})\nonumber\\
    &=1-q^2+\sum_{r\geq 1}(q^{6r^2+4r}+q^{6r^2+2r-2}-q^{6r^2+10r+2}-q^{6r^2-4r}).
\end{align}
Here we replaced $r$ by $r-1$ for the second sum.

For the series $D(q)$, we have
\begin{align}\label{D-exp}
   & D(q)=(1-q^2)\Big(\sum_{n\geq 1}\frac{q^{9n}(q^{6n^2-n}-q^{6n^2+5n+1})}{1-q^{6n+1}}  +\sum_{n\geq 1}\frac{q^{9n-3}(-q^{6n^2-5n+1}+q^{6n^2+n})}{1-q^{6n-1}}\Big)\nonumber\\
    &=(1-q^2)\sum_{n\geq 1}(q^{6n^2+8n}-q^{6n^2+4n-2}).
\end{align}

For the series $G_{a,b}(q)$, we have
\begin{align}
 &G_{1,0}(q)=G_{1,1}(q)=G_{1,2}(q)=G_{0,2}(q)=G_{2,2}(q)=0, \label{G-zero} \\
   & G_{0,0}(q)=\sum_{n,r\geq 1}\frac{q^{18nr+9n+3r}(q^{6n^2-n}-q^{6n^2+5n+1})(q^{6r^2+r}-q^{6r^2+7r+2})}{1-q^{6n+1}}\nonumber\\
    &=\sum_{n,r\geq 1}q^{6n^2+6r^2+18nr+8n+4r}(1-q^{6r+2}), \label{G00}\\
   & G_{0,1}(q)=\sum_{\substack{n\geq 1\\r\geq 0}}\frac{q^{18nr+15n+3r+1}(q^{6n^2-n}-q^{6n^2+5n+1})(-q^{6r^2+5r+1}+q^{6r^2+11r+5})}{1-q^{6n+1}}\nonumber\\
    &=-\sum_{\substack{n\geq 1\\r\geq 0}}q^{6n^2+6r^2+18nr+14n+8r+2}(1-q^{6r+4}) \nonumber \\
    &=-\sum_{n,r\geq 1} q^{6n^2+6r^2+18nr-4n-4r}(1-q^{6r-2}), \label{G01}\\
   & G_{2,0}(q)=\sum_{n,r\geq 1}\frac{q^{18nr+9n-3r-3}(-q^{6n^2-5n+1}+q^{6n^2+n})(q^{6r^2+r}-q^{6r^2+7r+2})}{1-q^{6n-1}}\nonumber\\
    &=-\sum_{n,r\geq 1}q^{6n^2+6r^2+18nr+4n-2r-2}(1-q^{6r+2}), \label{G20}\\
    &G_{2,1}(q)=\sum_{\substack{n\geq 1\\r\geq 0}}\frac{q^{18nr+15n-3r-4}(-q^{6n^2-5n+1}+q^{6n^2+n})(-q^{6r^2+5r+1}+q^{6r^2+11r+5})}{1-q^{6n-1}}\nonumber\\
    &=\sum_{\substack{n\geq 1\\r\geq 0}}q^{6n^2+6r^2+18nr+10n+2r-2}(1-q^{6r+4}) \nonumber \\
    &=\sum_{n,r\geq 1} q^{6n^2+6r^2+18nr-8n-10r+2}(1-q^{6r-2})\label{G21}
\end{align}

Substituting \eqref{C-exp}--\eqref{G21} into \eqref{T-CDG} and replacing $q$ by $q^{\frac{1}{2}}$, we deduce that
\begin{align}\label{T-wT}
&T(q^{\frac{1}{2}})=1-q+\sum_{r\geq 1}\Big(q^{3r^2+2r}+q^{3r^2+r-1}-q^{3r^2+5r+1}-q^{3r^2-2r} \nonumber \\
&\quad +(1-q)(q^{3r^2+4r}-q^{3r^2+2r-1})\Big)+\widetilde{T}(q),
\end{align}
where
\begin{align}
  &  \widetilde{T}(q)=\sum_{n,r\geq 1} q^{3n^2+3r^2+9nr+2r+4n}(1-q^{3r+1}) -\sum_{n,r\geq 1} q^{3n^2+3r^2+9nr-2n-2r}(1-q^{3r-1})\nonumber \\
    & \quad -\sum_{n,r\geq 1}q^{3n^2+3r^2+9nr-r+2n-1}(1-q^{3r+1})  +\sum_{n,r\geq 1} q^{3n^2+3r^2+9nr-4n-5r+1}(1-q^{3r-1}).
\end{align}
Regrouping the terms and replacing $(n,r)$ by $(-n,-r)$ for some of the sums, we have
\begin{align}
&\widetilde{T}(q)
=\Big(\sum_{n,r\geq 1}-\sum_{n,r\leq -1}\Big)q^{3n^2+3r^2+9nr+2r+4n}-
q\Big(\sum_{n,r\geq 1}-\sum_{n,r\leq -1}\Big)q^{3n^2+3r^2+9nr+5r+4n} \nonumber \\
&\quad +q^{-1}\Big(\sum_{n,r\geq 1}-\sum_{n,r\leq -1}\Big)q^{3n^2+3r^2+9nr-2n+r}+\Big(\sum_{n,r\geq 1}-\sum_{n,r\leq -1}\Big)q^{3n^2+3r^2+9nr+2n+2r} \nonumber \\
&=\Big(f_{2,3,2}(-q^7,-q^5,q^3)-\sum_{n=1}^\infty q^{3n^2+4n}-\sum_{r=1}^\infty q^{3r^2+2r}-1\Big) \nonumber \\
&\quad -q\Big(f_{2,3,2}(-q^7,-q^8,q^3)-\sum_{n=1}^\infty q^{3n^2+4n}-\sum_{r=1}^\infty q^{3r^2+5r}-1  \Big) \nonumber \\
&\quad +q^{-1}\Big(f_{2,3,2}(-q,-q^4,q^3)-\sum_{n=1}^\infty q^{3n^2-2n}-\sum_{r=1}^\infty q^{3r^2+r}-1 \Big) \nonumber \\
&\quad +\Big(f_{2,3,2}(-q^5,-q^5,q^3)-\sum_{n=1}^\infty q^{3n^2+2n}-\sum_{r=1}^\infty q^{3r^2+2r}-1  \Big) \label{wT-exp}
\end{align}
Substituting \eqref{wT-exp} into \eqref{T-wT}, we deduce that
\begin{align}\label{T-T1T2}
T(q^{\frac{1}{2}})=T_1(q)+T_2(q)
\end{align}
where
\begin{align}
    T_1(q)&=f_{2,3,2}(-q^7,-q^5,q^3)-qf_{2,3,2}(-q^7,-q^8,q^3)+q^{-1}f_{2,3,2}(-q,-q^4,q^3) \nonumber \\
    &\quad \quad +f_{2,3,2}(-q^5,-q^5,q^3), \label{T1-start}\\
    T_2(q)&=-\sum_{n=-\infty}^\infty q^{3n^2+2n}-\sum_{n=-\infty}^\infty q^{3n^2-2n-1}=-(1+q^{-1})\overline{J}_{1,6}. \label{T2-start}
\end{align}
This proves \eqref{T-lem}.
\end{proof}

Starting from Lemma \ref{lem-ST}, we have two different ways to complete the proof of Theorem \ref{thm-main}. The first method is to transform Hecke-type series to Appell--Lerch sums, and then show that they agree with the desired modular products. The second method is to transform the Hecek-type series to some other Hecke-type series whose modularity are known. Though the second proof is shorter, the method in the first proof exhibits greater versatility.

For the first proof, we establish some auxiliary identities to express certain combinations of $m(x,q,z)$ as infinite products.
\begin{lemma}\label{lem-W}
We have
\begin{align}
&W_1(q):=m(-q^{14},q^{30},q^6)+m(-q^{16},q^{30},q^{-8}) \nonumber \\
&=m(-q^{14},q^{30},q^6)+1-m(-q^{14},q^{30},q^8) =1-q^6\frac{J_{30}^3J_{2,30}\overline{J}_{28,30}}{J_{8,30}J_{6,30}\overline{J}_{22,30}\overline{J}_{20,30}}, \label{W1} \\
&W_2(q):=m(-q^{14},q^{30},q^{-6})+m(-q^{16},q^{30},q^{-8}) \nonumber \\
&=m(-q^{14},q^{30},q^{-6})+1-m(-q^{14},q^{30},q^8)=1+\frac{J_{30}^3J_{14,30}\overline{J}_{16,30}}{J_{8,30}J_{6,30}\overline{J}_{22,30}\overline{J}_{8,30}}, \label{W2} \\
&W_3(q):=q^{-1}m(-q^{-1},q^{30},q^6)+m(-q,q^{30},q^{-8}) \nonumber \\
&=-m(-q,q^{30},q^{-6})+m(-q,q^{30},q^{-8})=-q^5\frac{J_{30}^3J_{2,30}\overline{J}_{13,30}}{J_{6,30}J_{8,30}\overline{J}_{5,30}\overline{J}_{7,30}}, \label{W3} \\
&W_4(q):=q^{-1}m(-q^{-1},q^{30},q^{-6})+m(-q,q^{30},q^{-8}) \nonumber \\
&=-m(-q,q^{30},q^{6})+m(-q,q^{30},q^{-8})=q^6\frac{J_{30}^3J_{14,30}\overline{J}_{1,30}}{J_{6,30}{J}_{8,30}\overline{J}_{7,30}^2}, \label{W4} \\
&W_5(q):=q^{-4}m(-q^{-4},q^{30},q^8)+m(-q^4,q^{30},q^{12}) \nonumber \\
&=-m(-q^4,q^{30},q^{-8})+m(-q^4,q^{30},q^{12})=-q^{4}\frac{J_{30}^3J_{20,30}\overline{J}_{8,30}}{J_{8,30}J_{12,30}\overline{J}_{4,30}\overline{J}_{16,30}}, \label{W5} \\
&W_6(q):=q^{-4}m(-q^{-4},q^{30},q^8)+m(-q^4,q^{30},q^{-12}) \nonumber \\
& =-m(-q^4,q^{30},q^{-8})+m(-q^4,q^{30},q^{-12})=-q^{4}\frac{J_{30}^3J_{4,30}\overline{J}_{16,30}}{J_{8,30}J_{12,30}\overline{J}_{4,30}\overline{J}_{8,30}}, \label{W6} \\
&W_7(q):=m(-q^{11},q^{30},q^8)+m(-q^{19},q^{30},q^{-12})  \nonumber \\
&=m(-q^{11},q^{30},q^8)+1-m(-q^{11},q^{30},q^{12})=1-q^{7}\frac{J_{30}^3J_{4,30}\overline{J}_{1,30}}{J_{8,30}J_{12,30}\overline{J}_{19,30}\overline{J}_{23,30}}, \label{W7} \\
&W_8(q):=m(-q^{11},q^{30},q^8)+q^{-11}m(-q^{-11},q^{30},q^{12}) \nonumber \\
&=m(-q^{11},q^{30},q^8)-m(-q^{11},q^{30},q^{-12})=-q\frac{J_{30}^3J_{20,30}\overline{J}_{7,30}}{J_{12,30}J_{8,30}\overline{J}_{1,30}\overline{J}_{19,30}}. \label{W8}
\end{align}
\end{lemma}
\begin{proof}
By \eqref{m-id-4} we have
\begin{align}
&W_1(q)=m(-q^{14},q^{30},q^6)+1-m(-q^{14},q^{30},q^8)   \nonumber \\
&=1+q^8\frac{J_{30}^3j(q^{-2};q^{30})j(-q^{28};q^{30}) }{j(q^8;q^{30})j(q^6;q^{30})j(-q^{22};q^{30})j(-q^{20};q^{30})}.
\end{align}
Here the last equality follows from \eqref{m-minus} with $(x,q,z_0,z_1)\rightarrow (-q^{14},q^{30},q^8,q^6)$. This proves \eqref{W1}.

Similarly, the first equalities in \eqref{W2} and \eqref{W7} follow from \eqref{m-id-4}, and the first equalities in \eqref{W3}--\eqref{W6} and \eqref{W8} follow from \eqref{m-id-2}. The second equalities in \eqref{W2}--\eqref{W8} follow directly from \eqref{m-minus}.
\end{proof}

\begin{lemma}\label{lem-M}
We have
\begin{align}
  & M_1(q):= m(-q^{17},q^{30},q^2)+m(-q^{13},q^{30},q^6) \nonumber \\
  &=1-m(-q^{13},q^{30},q^{-2})+m(-q^{13},q^{30},q^6)=1-\frac{J_{30}^3J_{8,30}\overline{J}_{17,30}}{J_{2,30}J_{6,30}\overline{J}_{11,30}\overline{J}_{19,30}}, \label{M1} \\
 &  M_2(q):= m(-q^7,q^{30},q^{-2})+q^{-7}m(-q^{-7},q^{30},q^{12}) \nonumber \\
  &=m(-q^7,q^{30},q^{-2})-m(-q^7,q^{30},q^{-12})=\frac{J_{30}^3J_{10,30}\overline{J}_{7,30}}{J_{2,30}J_{12,30}\overline{J}_{5,30}^2},  \label{M2} \\
  & M_3(q):= m(-q^2,q^{30},q^2)+q^{-2}m(-q^{-2},q^{30},q^6) \nonumber \\
  &=m(-q^2,q^{30},q^2)-m(-q^2,q^{30},q^{-6})=-q^2\frac{J_{30}^3J_{8,30}\overline{J}_{2,30}}{J_{2,30}J_{6,30}\overline{J}_{4,30}^2},  \label{M3} \\
 & M_4(q):=  q^{-8}m(-q^{-8},q^{30},q^{-2})+m(-q^8,q^{30},q^{12}) \nonumber \\
 &=-m(-q^8,q^{30},q^2)+m(-q^8,q^{30},q^{12})=q^2\frac{J_{30}^3J_{10,30}\overline{J}_{22,30}}{J_{2,30}J_{12,30}\overline{J}_{10,30}^2}, \label{M4} \\
 &M_5(q):=-m(-q^8,q^{30},q^8)+m(-q^8,q^{30},q^{-12})=q^4\frac{J_{30}^3J_{20,30}\overline{J}_{4,30}}{J_{8,30}J_{12,30}\overline{J}_{4,30}\overline{J}_{16,30}},  \label{M5} \\
 &M_6(q):=-m(-q^{13},q^{30},q^{-8})+m(-q^{13},q^{30},q^{-6})=q^5\frac{J_{30}^3J_{2,30}\overline{J}_{1,30}}{J_{6,30}J_{8,30}\overline{J}_{5,30}\overline{J}_{7,30}}, \label{M6} \\
 &M_7(q):=-q^{-7}m(-q^{-7},q^{30},q^8)+m(-q^{23},q^{30},q^{-12}) \nonumber \\
 &=1-m(-q^{23},q^{30},q^8)+m(-q^{23},q^{30},q^{-12})=1+q\frac{J_{30}^3J_{20,30}\overline{J}_{19,30}}{J_{8,30}J_{12,30}\overline{J}_{1,30}\overline{J}_{11,30}},  \label{M7} \\
& M_8(q):=-m(-q^{-2},q^{30},q^{-8})+m(-q^{-2},q^{30},q^{-6})=q^8\frac{J_{30}^3J_{2,30}\overline{J}_{16,30}}{J_{6,30}J_{8,30}\overline{J}_{10,30}\overline{J}_{8,30}}. \label{M8}
\end{align}
\end{lemma}
\begin{proof}
The first equalities in \eqref{M1} and \eqref{M7} follow from \eqref{m-id-4} and \eqref{m-id-3}, respectively.  The first equalities in \eqref{M2}--\eqref{M4} follow from \eqref{m-id-2}. The remaining equalities all follow directly from \eqref{m-minus}.
\end{proof}

Based on Lemmas \ref{lem-ST}--\ref{lem-M}, we are now ready to give our first proof for Theorem \ref{thm-main}.
\begin{proof}[First Proof of Theorem \ref{thm-main}]
(1) Applying Lemma \ref{f-m-expan} with $\ell=1$, we deduce that
\begin{align}
&f_{2,3,2}(-q^4,-q^4,q^3)=\overline{J}_{2,6}m(-q^{14},q^{30},q^6)+\overline{J}_{2,6}m(-q^{14},q^{30},q^{-6}) \nonumber \\
&+q^{-4}\overline{J}_{1,6}m(-q^{-1},q^{30},q^6)+q^{-4}\overline{J}_{1,6}m(-q^{-1},q^{30},q^{-6}),  \label{f-1}\\
&f_{2,3,2}(-q^4,-q^5,q^3)=\overline{J}_{1,6}m(-q^{11},q^{30},q^8)+\overline{J}_{2,6}m(-q^{16},q^{30},q^{-8}) \nonumber \\
&+q^{-6}\overline{J}_{2,6}m(-q^{-4},q^{30},q^8)+q^{-3}\overline{J}_{1,6}m(-q,q^{30},q^{-8}),  \label{f-2}\\
&f_{2,3,2}(-q^5,-q^8,q^3)=q^{-2}\overline{J}_{2,6}m(-q^4,q^{30},q^{12})+\overline{J}_{1,6}m(-q^{19},q^{30},q^{-12}) \nonumber  \\
&+q^{-11}\overline{J}_{1,6}m(-q^{-11},q^{30},q^{12})+q^{-2}\overline{J}_{2,6}m(-q^4,q^{30},q^{-12}). \label{f-3}
\end{align}
Substituting \eqref{f-1}--\eqref{f-3} into \eqref{S1-defn} and collecting similar terms, we have
\begin{align}\label{S1-W}
S_1(q)&=\overline{J}_{2,6}\Big(W_1(q)+W_2(q)+q^{-2}W_5(q)+q^{-2}W_6(q)\Big) \nonumber \\
& \quad +\overline{J}_{1,6}\Big(q^{-3}W_3(q)+q^{-3}W_4(q)+W_7(q)+W_8(q)\Big).
\end{align}
Substituting \eqref{W1}--\eqref{W8} into \eqref{S1-W}, and then substituting the result and \eqref{S2-defn} into \eqref{S-S1S2}, we deduce that
\begin{align}\label{S-proof-product}
S(q^{\frac{1}{2}})&=\overline{J}_{2,6}\Big(-q^6\frac{J_{30}^3J_{2,30}\overline{J}_{28,30}}{J_{8,30}J_{6,30}\overline{J}_{22,30}\overline{J}_{20,30}}+ \frac{J_{30}^3J_{14,30}\overline{J}_{16,30}}{J_{8,30}J_{6,30}\overline{J}_{22,30}\overline{J}_{8,30}} \nonumber \\
&\quad -q^{2}\frac{J_{30}^3J_{20,30}\overline{J}_{8,30}}{J_{8,30}J_{12,30}\overline{J}_{4,30}\overline{J}_{16,30}} -q^{2}\frac{J_{30}^3J_{4,30}\overline{J}_{16,30}}{J_{8,30}J_{12,30}\overline{J}_{4,30}\overline{J}_{8,30}}\Big) \nonumber \\
&\quad +\overline{J}_{1,6}\Big(-q^2\frac{J_{30}^3J_{2,30}\overline{J}_{13,30}}{J_{6,30}J_{8,30}\overline{J}_{5,30}\overline{J}_{7,30}}+ q^3\frac{J_{30}^3J_{14,30}\overline{J}_{1,30}}{J_{6,30}{J}_{8,30}\overline{J}_{7,30}^2} \nonumber \\
&\quad -q^{7}\frac{J_{30}^3J_{4,30}\overline{J}_{1,30}}{J_{8,30}J_{12,30}\overline{J}_{19,30}\overline{J}_{23,30}}-q\frac{J_{30}^3J_{20,30}\overline{J}_{7,30}}{J_{12,30}J_{8,30}\overline{J}_{1,30}\overline{J}_{19,30}} \Big).
\end{align}
Using the Maple approach in \cite{Frye-Garvan}, it is easy and straightforward to prove \eqref{S-product} from \eqref{S-proof-product}. Therefore we prove \eqref{eq-conj-1}.

(2) Applying Lemma \ref{f-m-expan} with $\ell=1$, we deduce that
\begin{align}
&f_{2,3,2}(-q^7,-q^5,q^3)=\overline{J}_{1,6}m(-q^{17},q^{30},q^2)+q^{-1}\overline{J}_{1,6}m(-q^7,q^{30},q^{-2}) \nonumber \\
&\qquad +q^{-3}\overline{J}_{2,6}m(-q^2,q^{30},q^2)+q^{-9}\overline{J}_{2,6}m(-q^{-8},q^{30},q^{-2}), \label{f-4} \\
&f_{2,3,2}(-q^7,-q^8,q^3)=q^{-2}\overline{J}_{2,6}m(-q^8,q^{30},q^8)+q^{-1}\overline{J}_{1,6}m(-q^{13},q^{30},q^{-8}) \nonumber \\
&\qquad +q^{-9}\overline{J}_{5,6}m(-q^{-7},q^{30},q^8)+q^{-6}\overline{J}_{2,6}m(-q^{-2},q^{30},q^{-8}),  \label{f-5} \\
&f_{2,3,2}(-q,-q^4,q^3)=\overline{J}_{2,6}m(-q^8,q^{30},q^{12})+\overline{J}_{1,6}m(-q^{23},q^{30},q^{-12})\nonumber \\
&\qquad +q^{-7}\overline{J}_{1,6}m(-q^{-7},q^{30},q^{12})+\overline{J}_{2,6}m(-q^8,q^{30},q^{-12}),  \label{f-6}\\
&f_{2,3,2}(-q^5,-q^5,q^3)=\overline{J}_{1,6}m(-q^{13},q^{30},q^6)+\overline{J}_{1,6}m(-q^{13},q^{30},q^{-6})\nonumber \\
&\qquad +q^{-5}\overline{J}_{2,6}m(-q^{-2},q^{30},q^6)+q^{-5}\overline{J}_{2,6}m(-q^{-2},q^{30},q^{-6}). \label{f-7}
\end{align}
Substituting \eqref{f-4}--\eqref{f-7} into \eqref{T1-start}, and then collecting similar terms, we deduce that
\begin{align}\label{T1-M}
    T_1(q)&=\overline{J}_{1,6}\Big(M_1(q)+q^{-1}M_2(q)+M_6(q)+q^{-1}M_7(q)\Big) \nonumber \\
    &\quad +\overline{J}_{2,6}\Big(q^{-3}M_3(q)+q^{-1}M_4(q)+q^{-1}M_5(q)+q^{-5}M_8(q) \Big).
\end{align}
Substituting \eqref{M1}--\eqref{M8} into \eqref{T1-M}, and then substituting the result and \eqref{T2-start} into \eqref{T-T1T2}, we deduce that
\begin{align}\label{T-proof-product}
    T(q^{\frac{1}{2}})&=\overline{J}_{1,6}\Big(-\frac{J_{30}^3J_{8,30}\overline{J}_{17,30}}{J_{2,30}J_{6,30}\overline{J}_{11,30}\overline{J}_{19,30}}+q^{-1}\frac{J_{30}^3J_{10,30}\overline{J}_{7,30}}{J_{2,30}J_{12,30}\overline{J}_{5,30}^2} +q^5\frac{J_{30}^3J_{2,30}\overline{J}_{1,30}}{J_{6,30}J_{8,30}\overline{J}_{5,30}\overline{J}_{7,30}} \nonumber \\
    &\quad +\frac{J_{30}^3J_{20,30}\overline{J}_{19,30}}{J_{8,30}J_{12,30}\overline{J}_{1,30}\overline{J}_{11,30}} \Big)  +\overline{J}_{2,6}\Big(-q^{-1}\frac{J_{30}^3J_{8,30}\overline{J}_{2,30}}{J_{2,30}J_{6,30}\overline{J}_{4,30}^2}+q\frac{J_{30}^3J_{10,30}\overline{J}_{22,30}}{J_{2,30}J_{12,30}\overline{J}_{10,30}^2}
    \nonumber \\
    &\quad +q^3\frac{J_{30}^3J_{20,30}\overline{J}_{4,30}}{J_{8,30}J_{12,30}\overline{J}_{4,30}\overline{J}_{16,30}}+q^3\frac{J_{30}^3J_{2,30}\overline{J}_{16,30}}{J_{6,30}J_{8,30}\overline{J}_{10,30}\overline{J}_{8,30}}  \Big).
\end{align}

Using the Maple approach in \cite{Frye-Garvan}, it is easy and straightforward to prove \eqref{T-product} from \eqref{T-proof-product}. Therefore we prove \eqref{eq-conj-2}.
\end{proof}

It is possible to prove Theorem \ref{thm-main}
 without using $m(x,q,z)$. Here we give another proof starting from Lemma \ref{lem-ST}. This proof is based on the following identities in the work of Kim and Lovejoy \cite[pp. 759--760]{Kim-Lovejoy}:
\begin{align}
    \frac{(q;q)_\infty^2}{(q,q^4;q^5)_\infty}&=f_{6,9,6}(-q^5,-q^4,q)-qf_{6,9,6}(-q^7,-q^7,q)\nonumber \\
    & \quad -q^2f_{6,9,6}(-q^8,-q^{10},q)+q^4f_{6,9,6}(-q^{10},-q^{13},q), \label{KL-1} \\
    \frac{(q;q)_\infty^2}{(q^2,q^3;q^5)_\infty}&=f_{6,9,6}(-q^5,-q^5,q)-qf_{6,9,6}(-q^8,-q^7,q)\nonumber \\
    &\quad -qf_{6,9,6}(-q^7,-q^8,q)+q^3f_{6,9,6}(-q^{10},-q^{10},q). \label{KL-2}
\end{align}
By employing the Bailey machinery and the Andrews--Gordon identities,  Kim and Lovejoy \cite[Corollary 4.1]{Kim-Lovejoy} actually proved a more general formula which contains \eqref{KL-1} and \eqref{KL-2} as special instances.

\begin{proof}[Second proof of Theorem \ref{thm-main}]
(1) In view of \eqref{S-lem} and \eqref{KL-1} and using \eqref{f-base}, the desired identity \eqref{S-product} is equivalent to
\begin{align}\label{f-equivalent-1}
&f_{2,3,2}(-q^4,-q^4,q^3)+2f_{2,3,2}(-q^4,-q^5,q^3)+f_{2,3,2}(-q^5,-q^8,q^3)-\overline{J}_{1,6}-2\overline{J}_{2,6} \nonumber \\
&=f_{2,3,2}(-q^5,-q^4,q^3)-qf_{2,3,2}(-q^7,-q^7,q^3)\nonumber \\
    & \quad -q^2f_{2,3,2}(-q^8,-q^{10},q^3)+q^4f_{2,3,2}(-q^{10},-q^{13},q^3).
\end{align}
By \eqref{f-symmetry} we have
\begin{align}
    f_{2,3,2}(-q^4,-q^5,q^3)=f_{2,3,2}(-q^5,-q^4,q^3). \label{f-equal}
\end{align}

Setting $(a,b,c,x,y,q)\rightarrow (2,3,2,-q^4,-q^4,q^3)$ in \eqref{f-id-1} we deduce that
\begin{align}
    f_{2,3,2}(-q^4,-q^4,q^3)=q^4f_{2,3,2}(-q^{13},-q^{10},q^3)+\overline{J}_{2,6}. \label{2nd-1}
\end{align}
By \eqref{f-id-0} and \eqref{f-id-1} with $(a,b,c,x,y,q)\rightarrow (2,3,2,-q^4,-q^5,q^3)$ we deduce that
\begin{align}
   & f_{2,3,2}(-q^4,-q^5,q^3)+q^2f_{2,3,2}(-q^8,-q^{10},q^3) \nonumber \\
   &=f_{2,3,2}(-q^4,-q^5,q^3)-q^5f_{2,3,2}(-q^{13},-q^{11},q^3)=\overline{J}_{2,6}. \label{2nd-2}
\end{align}
By \eqref{f-id-0} and \eqref{f-id-1} with $(a,b,c,x,y,q)\rightarrow (2,3,2,-q^5,-q^8,q^3)$ we deduce that
\begin{align}
   & f_{2,3,2}(-q^5,-q^8,q^3)+qf_{2,3,2}(-q^7,-q^7,q^3) \nonumber \\
   &= f_{2,3,2}(-q^5,-q^8,q^3)-q^8f_{2,3,2}(-q^{14},-q^{14},q^3)=\overline{J}_{1,6}. \label{2nd-3}
\end{align}
Combining \eqref{f-equal}--\eqref{2nd-3}, we obtain \eqref{f-equivalent-1} and hence prove \eqref{S-product}.

(2) In view of \eqref{T-lem} and \eqref{KL-2} and using \eqref{f-base}, the identity \eqref{T-product} is equivalent to
\begin{align}
&f_{2,3,2}(-q^7,-q^5,q^3)+f_{2,3,2}(-q^5,-q^5,q^3)+q^{-1}f_{2,3,2}(-q,-q^4,q^3)\nonumber \\
&\quad -qf_{2,3,2}(-q^7,-q^8,q^3)-(1+q^{-1})\overline{J}_{1,6} \nonumber \\
&=f_{2,3,2}(-q^{5},-q^{5},q^3)-qf_{2,3,2}(-q^{8},-q^{7},q^3)\nonumber \\
    &\quad -qf_{2,3,2}(-q^{7},-q^{8},q^3)+q^3f_{2,3,2}(-q^{10},-q^{10},q^3).
\end{align}
Canceling the common terms on both sides reduces it to
\begin{align}\label{f-equivalent-2}
&f_{2,3,2}(-q^7,-q^5,q^3)+q^{-1}f_{2,3,2}(-q,-q^4,q^3)-(1+q^{-1})\overline{J}_{1,6} \nonumber \\
&=-qf_{2,3,2}(-q^7,-q^8,q^3)+q^3f_{2,3,2}(-q^{10},-q^{10},q^3).
\end{align}
By \eqref{f-id-1} with $(a,b,c,x,y,q)\rightarrow (2,3,2,-q,-q^4,q^3)$, we deduce that
\begin{align}\label{2nd-4}
    f_{2,3,2}(-q,-q^4,q^3)=q^4f_{2,3,2}(-q^{10},-q^{10},q^3)+\overline{J}_{1,6}.
\end{align}
By \eqref{f-id-0} and \eqref{f-id-1} with $(a,b,c,x,y,q)\rightarrow (2,3,2,-q^5,-q^7,q^3)$, we deduce that
\begin{align}\label{2nd-5}
   & f_{2,3,2}(-q^5,-q^7,q^3)=q^7f_{2,3,2}(-q^{14},-q^{13},q^3)+j(-q^5,q^6) \nonumber \\
    &=-qf_{2,3,2}(-q^7,-q^8,q^3)+\overline{J}_{1,6}.
\end{align}
Combining \eqref{2nd-4} and \eqref{2nd-5}, we obtain \eqref{f-equivalent-2} and hence prove \eqref{T-product}.
\end{proof}


\begin{thebibliography}{0}
\bibitem{Cao-Rosengren-Wang} Z. Cao, H. Rosengren and L. Wang, On some double Nahm sums of Zagier,  J. Combin. Theory Ser. A. 202 (2024), 105819.


\bibitem{Feigin} I. Cherednik and B. Feigin, Rogers--Ramanujan type identities and Nil-DAHA, Adv. Math. 248 (2013), 1050--1088.


\bibitem{Frye-Garvan} J. Frye and F.G. Garvan, Automatic proof of theta-function identities, in Elliptic
Integrals, Elliptic Functions and Modular Forms in Quantum Field Theory, Texts $\&$
Monographs in Symbolic Computation (Springer, Cham, 2019), pp. 195--258.

\bibitem{Hickerson-Mortenson} D.R. Hickerson and E.T. Mortenson, Hecke-type double sums, Appell--Lerch sums and mock theta functions, I, Proc. London Math. Soc. (3) 109 (2014), 382--422.

\bibitem{Kim-Lovejoy} B. Kim and J. Lovejoy, Ramanujan type partial theta  identities and conjugate Bailey pairs, \uppercase\expandafter{\romannumeral 2}. Multisums, Ramanujan J. 46 (2018), 743--764.

\bibitem{Lovejoy} J. Lovejoy, Ramanujan-type partial theta identities and conjugate Bailey pairs, Ramanujan J. 29 (2012), no. 1--3, 51--67.

\bibitem{Nahm07} W. Nahm, Conformal field theory and torsion elements of the Bloch group. In Frontiers in Number Theory, Physics, and Geometry II: On Conformal Field Theories, Discrete Groups and Renormalization, pp.\ 67--132. Springer, 2007.

\bibitem{Rogers1894} L.J. Rogers, Second memoir on the expansion of certain infinite products, Proc. London Math. Soc. 25 (1894), 318--343.

\bibitem{Slater} L.J. Slater, A new proof of Rogers' transformations of infinite series, Proc. Lond. Math. Soc. (2), 53 (1951), 460--475.

\bibitem{VZ}  M. Vlasenko and S. Zwegers, Nahm's conjecture: asymptotic computations and counterexamples, Commun. Number Theory Phys. 5(3) (2011), 617--642.


\bibitem{Wang-rank2} L. Wang, Identities on Zagier's rank two examples for Nahm's problem, Res. Math. Sci. 11 (2024), Art. 49.

\bibitem{Wang-rank3} L. Wang, Explicit forms and proofs of Zagier's rank three examples for Nahm's problem, Adv. Math. 450 (2024), Paper No. 109743.

\bibitem{Wang-Zeng} L. Wang and W. Zeng, Rogers--Ramanujan type identities for partial Nahm sums, arXiv: 2502.19309.

\bibitem{Zagier} D. Zagier, The dilogarithm function, in Frontiers in Number Theory, Physics and Geometry, II, Springer, 2007, 3--65.
\end{thebibliography}
\end{document}